\documentclass[11pt]{amsart}
\usepackage{amsmath,amssymb,latexsym,soul,cite,mathrsfs}

\usepackage{color,enumitem,graphicx}
\usepackage[colorlinks=true,urlcolor=blue,
citecolor=red,linkcolor=blue,linktocpage,pdfpagelabels,
bookmarksnumbered,bookmarksopen]{hyperref}
\usepackage[english]{babel}

\usepackage[left=2.9cm,right=2.9cm,top=2.8cm,bottom=2.8cm]{geometry}
\usepackage[hyperpageref]{backref}

\usepackage[colorinlistoftodos]{todonotes}
\makeatletter
\providecommand\@dotsep{5}
\def\listtodoname{List of Todos}
\def\listoftodos{\@starttoc{tdo}\listtodoname}
\makeatother

\numberwithin{equation}{section}
\def\R {{\rm I}\hskip -0.85mm{\rm R}}
\def\N {{\rm I}\hskip -0.85mm{\rm N}}

\newtheorem{theorem}{Theorem}[section]
\newtheorem{proposition}[theorem]{Proposition}
\newtheorem{lemma}[theorem]{Lemma}

\newtheorem{remark}{Remark}

\title[On a generalized timoshenko-kirchhoff equation]
{ On a generalized timoshenko-kirchhoff equation \\ with sublinear nonlinearities}

\author[J. R. Santos Jr.]{Jo\~ao R. Santos Junior}
\author[G. Siciliano]{Gaetano Siciliano}

\address[J. R. Santos Jr.]{\newline\indent Faculdade de Matem\'atica
\newline\indent 
Instituto de Ci\^{e}ncias Exatas e Naturais
\newline\indent 
Universidade Federal do Par\'a
\newline\indent
Avenida Augusto corr\^{e}a 01, 66075-110, Bel\'em, PA, Brazil}
\email{\href{mailto: joaojunior@ufpa.br }{joaojunior@ufpa.br}}
\address[G. Siciliano]{\newline\indent Departamento de Matem\'atica
\newline\indent 
Instituto de Matem\'atica e Estat\'istica
\newline\indent 
 Universidade de S\~ao Paulo 
\newline\indent 
Rua do Mat\~ao 1010,  05508-090, S\~ao Paulo, SP, Brazil }
\email{\href{mailto:sicilian@ime.usp.br}{sicilian@ime.usp.br}}

\thanks{Jo\~ao R. Santos was partially
supported by CNPq-Proc. 302698/2015-9, Brazil.
Gaetano Siciliano  was partially supported by
Fapesp and CNPq, Brazil. }
\subjclass[2010]{ 35J30, 	%Higher-order elliptic equations
35J50,  	%Variational methods for elliptic systems
35J57, %Boundary value problems for second-order elliptic systems
	47H10  	%Fixed-point theorems 
	}
\keywords{ Timoshenko-Kirchhoff type equation, sublinear nonlinearity, topological and variational methods.}

\begin{document}

\maketitle
\begin{abstract}
In this paper we consider a generalized fourth order nonlinear Kirchhoff equation in a bounded domain
in $\mathbb R^{N}, N\geq2$ under Navier boundary conditions and with sublinear nonlinearity.
We employ a change of variable which reduces the problem to a semilinear one.
Then variational and topological tools are used in order to prove the existence of a solution.
\end{abstract}
\maketitle

%------------------------------------------------------------------------------
\section{Introduction}

%------------------------------------------------------------------------------

This paper concerns with the existence  of solutions $u:\Omega\to \R$ to the following nonlocal problem 
\begin{equation}\label{P}\tag{P}
\left \{ \begin{array}{ll}
\Delta^{2}u-\text{div}\left(m\left (u, |\nabla u|_{2}^{2}\right)\nabla u\right)= f(x, u) & \mbox{in $\Omega$,}\\
u=\Delta u=0 & \mbox{on $\partial\Omega$,}
\end{array}\right.
\end{equation}
where $\Omega$ is a bounded and smooth domain in $\R^{N}, N\geq2$
and
$$m:\R\times[0,\infty) \to \R, \quad f:\Omega\times \R\to \R$$ 
are given data satisfying  suitable conditions.
Hereafter we denote with  $|\cdot|_{p}$  the usual $L^{p}(\Omega)-$norm.
We say here, once for all, that by a solution of the above problem
we mean a function $u_{*}\in H^{2}(\Omega)\cap H^{1}_{0}(\Omega)\cap L^{\infty}(\Omega)$
such that
$$
\int_{\Omega}\Delta u_{*}\Delta\varphi\,dx+\int_{\Omega} m(u_{*}(x) , |\nabla u_{*}|_{2}^{2}) \nabla u_{*} \nabla \varphi \,dx = \int_{\Omega}f(x, u_{*})\varphi \, dx, \ \forall \varphi\in H^{2}(\Omega)\cap H^{1}_{0}(\Omega).
$$
Our assumptions on $m$ and $f$ will guarantee that the above identity makes sense.

\medskip

The equation under study is a slight generalization
of the following one 
\begin{equation}\label{TK}%\tag{KP}
\left \{ \begin{array}{ll}
\Delta^{2}u -m\left( |\nabla u|_{2}^{2}\right)\Delta u= f(x, u) & \mbox{in $\Omega$,}\\
\Delta u=u=0 & \mbox{on $\partial\Omega$,}
\end{array}\right.
\end{equation}
 known as 
{\sl Timoshenko-Kirchhoff } plate equation. Without entering in details
here, we say that problem \eqref{TK} appears in Nonlinear Solid Mechanics and Mechanics of Materials.
In particular it describes the stationary solutions of an elastic plate (in case $N=2$), with fixed boundary,
subjected to small transversal vibrations and taking into account $(i)$ the Kirchhoff correction
to the classical wave equation of D'Alambert, $(ii)$ 
the correction for rotary inertia of the cross section of the plate and $(iii)$
the influence of shearing  strains introduced by Timoshenko.
This model was proposed by Arosio in \cite{Arosio1998, Arosio1999, Arosio2001} to which we refer the reader for the deduction
of the equation and the study of the local well-posedness of the associated Cauchy problem.
Actually, although the model proposed by Arosio deals mainly with the $1$ dimensional case (the beam equation) 
he studies an Abstract Cauchy
 problem in which the $N-$dimensional case, i.e. problem \eqref{TK}, falls down, see  \cite[equation (3.2)]{Arosio1999}.
% hyperbolic  equation
%\begin{equation*}%\label{KE}\tag{KE}
%\rho\displaystyle\frac{\partial^{2}u}{\partial
%t^{2}}-\biggl(\displaystyle\frac{P_{0}}{h}+\displaystyle\frac{E}{2L}\int^{L}_{0}\biggl|\displaystyle\frac{\partial
%u}{\partial
%x}\biggl|^{2}dx\biggl)\displaystyle\frac{\partial^{2}u}{\partial
%x^{2}}=0,
%\end{equation*}
%known in the literature as the 

As we already said, problem \eqref{TK} is a generalization of the 
``pure'' Kirchhoff equation which involves the Kirchhoff operator
$K(u)=\left(a+b |\nabla u|_{2}^{2}\right)\Delta u$.
 The  Kirchhoff equation  has been very studied
and is almost impossible to give an
exhaustive bibliography. 
However less results are present in the literature involving the 
operator $\Delta^{2}u +m(|\nabla u|_{2}^{2})\Delta u$.
We just cite the recent papers
 \cite{CF,cinesi} (see also the references therein)
 where  the authors study the existence of solutions for  an equation of type
$$\Delta^{2} u +m( |\nabla u|_{2}^{2})\Delta u +V(x)u=f(x,u),\ \  u\in H^{2}(\R^{N})$$
with suitable assumptions on $f$ and $V$.

\medskip

Summing up, the operator appearing in our problem \eqref{P}  can be seen  as a generalization of the 
Kirchhoff operator $K(u)=\left(a+b |\nabla u|_{2}^{2}\right)\Delta u$;
indeed beside the  correction with a fourth order operator
we are also assuming
 that the velocity of the displacement of the membrane
 is proportional to the gradient of the displacement with a factor $m$ depending not only
 on the variation of the ``superficial area'' of $\Omega$, but even on the same displacement:
$$
\textbf v=-m\left(u, |\nabla u|^{2}_{2} \right)\nabla u.
$$
Such a model is quite reasonable and  
there are also physical situations where rather than on $|\nabla u|_{2}$
the function $m$ depends on other quantities such as the $L^{1}$-norm 
of $u$, see e.g. \cite{CR}.
However,  to the best of our knowledge, 
there are few results concerning the case $m(u,|\nabla u|_{2}^{2})$.
We mention the recent paper \cite{SS} where the problem
$$-\text{div}(m( |\nabla u|_{2}^{2})\nabla u)=f(x,u)\ \text{in } \Omega, \quad u=0 \ \text{on }\partial\Omega$$
has been studied in a bounded domain. 
However in \cite{SS} one more assumptions on the function
$m$ with respect to the present paper was given, see condition \eqref{m_{2}} in Lemma \ref{le} below. 
We believe that using the method of this paper the assumption \eqref{m_{2}} can be removed in 
\cite{SS}.
%Now
%we do not need  this condition, and the reason seems to be the fact that now we have the biharmonic
%operator. See also Remark \ref{rem:comentario} in the next section.}

%Here $L$ is the length of the string, $h$ is the area of the cross-section, $E$ is the Young modulus of the material, $\rho$ is the mass density and $P_{0}$ is the initial tension. 
%This equation was introduced in \cite{kirchhoff} as a more realistic model then the classical  d'Alembert's wave equation. It  describes small transverse vibrations of an elastic string with fixed ends taking into account 
%the fact that the length is changing in time during the vibrations, and consequently the tension is also changing as the time
%goes on.
%In fact, since the length of the string is variable during the 
%vibrations, the tension changes with the time and depends of the $L^{2}-$norm of the gradient of the displacement $u$. 
%Problem \eqref{KP} is  nonlocal in nature due to
%the presence of the term $\int_{\Omega}|\nabla u|^{2}dx$. 
%which implies that the equation in (\ref{KP}) is no longer a pointwise identity. 
%This phenomenon causes some mathematical difficulties which make the study of such class of problems particularly interesting. 
%In the mathematical literature there are many results on
% problem \eqref{KP}; the interested reader can consult e.g.\cite{ACM,MR,ma} and the references therein.
%Results treating the problem from a variational point of  view can be found, for instance, in \cite{AA, AFP, FISJ, HZ, PZ} and their references.

\bigskip

Coming back to the present paper,  we will treat problem \eqref{P} separately 
in the cases
\begin{itemize}
\item[a)] $f=f(x)$, and
\item[b)]  $f=f(x,u)$ with sublinear growth.
\end{itemize}
We prefer to start with the particular case a) since some results
of this case will be used also in the more general case b).

Before to state  our main results, let us specify the assumptions on $m$. 
We first introduce the following convention:
for every $r\geq0$, we denote  with $m_{r}$ the map
$$m_{r}: t\in\R\mapsto m(t,r) \in \R.$$
We  suppose that $m: \R\times [0,\infty)\to(0,\infty)$ is a function such that: \medskip
\begin{enumerate}[label=(m\arabic*),ref=m\arabic*,start=0]
\item\label{m_{0}} is  continuous;  \medskip
\item\label{m_{1}} there is $\mathfrak m >0$ such that $m(t, r)\geq \mathfrak m $ for all $t\in\R$ and $r\in[0,\infty)$;  
 \medskip
%\item\label{m_{2}} \textcolor{red}{TIRAR>>>>} for each $r\in [0,+\infty)$ the map 
%$m_{r}:\R\to(0,+\infty)$ is strictly decreasing in $(-\infty,0)$  and strictly increasing in $(0,+\infty)$. \medskip 
\end{enumerate}
%For example,
%$m(t,r)= t^{2}(r^{p}+1)+1$, with $p>0$, or 
%$m(t,r)= t^{2}(e^{t^{2} e^{r}}+1)+1$ satisfies these conditions.
%Moreover these functions achieve their minimum $\mathfrak m=1$
%in points of type $(0,r).$

Then we have the following
\begin{theorem}\label{th:main1}
If  \eqref{m_{0}}-\eqref{m_{1}} hold, $0\not\equiv f\in L^{q}(\Omega)$ and $q>N/2$,  
then problem \eqref{P} with $f=f(x)$, has a nontrivial weak solution $u_{*}$.
\end{theorem}

\noindent To deal with the general case  $f=f(x,u)$
we  introduce some notations.
Let $\lambda_{1}$ be the first eigenvalue of $-\Delta$ in $H^{1}_{0}(\Omega)$ and $\gamma$ denotes a positive constant (independent of $h\in L^{\infty}(\Omega)$ and $u\in C^{1,\alpha}(\overline\Omega)$, $0<\alpha<1$) such that
$$
\|u\|_{C^{1, \alpha}(\overline \Omega)}\leq \gamma |h|_{\infty},
$$
where
\begin{equation*}
\left \{ \begin{array}{ll}
-\Delta u= h(x) & \mbox{in $\Omega$,}\\
u=0 & \mbox{on $\partial\Omega$.}
\end{array}\right.
\end{equation*}
We then assume that $f:\Omega\times \R\to\R$ is  a Carath\'eodory function satisfying:
\begin{enumerate}[label=(f\arabic*),ref=f\arabic*,start=1]
\item \label{f_1}$f(x, 0)\neq 0$, \smallskip
 \item\label{f_2} there exists $\mu\in L^{\infty}(\Omega)$, $\delta\in(0,1]$ and $\nu\in (0, \min\{\lambda_{1}^{(1+3\delta)/2}/|\Omega|^{1-\delta}, \mathfrak m^{\delta}/\gamma\})$, such that 
  $$ 
  |f(x, t)|\leq \mu(x) + \nu |t|^{\delta }\quad \mbox{ a.e. in $\Omega$ and $\forall t\in \R$},
  $$
%  where $\delta\in (0,1]$.
  \item\label{f_3} there is $\theta\in (0, \lambda_{1}^{2})$ such that 
  $$
  |f(x, t_{1})-f(x, t_{2})|\leq \theta|t_{1}-t_{2}| \quad \mbox{a.e. in $\Omega$ and $\forall t_{1},t_{2}\in \R$}.
  $$
% for all $t_{1}, t_{2}\in \R$. 
  \end{enumerate}
%  \textcolor{red}{Actually, as it will be evident from the proof, condition \eqref{f_2} can be relaxed in the sense
%that if $\delta\in(0,1)$ we just need to assume on the parameter $\nu$ that it belongs to
%$(0,\mathfrak m^{\delta}/\gamma)$. We have written \eqref{f_2} in that form just for the
%sake of simplicity.}
  
Then we have the following
\begin{theorem}\label{th:main2}
If  \eqref{m_{0}}-\eqref{m_{1}} and \eqref{f_1}-\eqref{f_3} hold, 
then problem \eqref{P} with $f=f(x,u)$, has a nontrivial weak solution $u_{*}$.

\end{theorem}

Some comments now are in order. 

First of all, as we will see the weak solutions we find are indeed in $C^{1,\alpha}(\overline \Omega)$.
However, arguing as in \cite[Theorem 2.1]{BS}, these solutions are indeed classical.

Moreover  our approach in proving the results mixes variational and fixed points methods.
Indeed the strategy in proving both our theorems is
\begin{itemize}
\item[Step 1:] to find solutions for an auxiliary 
problem depending on a parameter $r\geq0$, let us say $u_{r}$; here variational tools are used.
\item[Step 2:] prove estimates on $u_{r}$ and obtain a solution of \eqref{P} by topological tools.
\end{itemize}
We aim also to show that a simple change of variable
(used also in \cite{SS}) can transform our original problem into a second order semilinear system
(this is used in Step 1). Of course the change of variables is independent on our assumptions
on $f$.
So we believe that these techniques can be used to treat
also other type of equations.

For what concern the function $m$, observe that in contrast to the case in which
$m$  depend only on $|\nabla u|_{2}$,  no assumptions
on the growth at infinity  of $m$ with respect to  $|\nabla u|_{2}^{2}$ are made
here.

Finally we leave as an interesting and open problem the case
in which $m$ might vanish, which correspond in some sense to a degenerate operator.

%\begin{remark}
%Here, based on a result in \cite{BO}, we will assume a sublinear assumption on $f$
%and we will use  topological tools. However, depending on other type of assumptions on the nonlinearity $f$,
%other methods, for instance variational, can be  employed to solve the local equation in which the problem is transformed after the change of variable (and then recover a solution of the original equation).
%\end{remark}

\medskip

%The main novelty of our approach is that the proofs are based on a simple ``change of variable'' device
%which seems not to have been used for these kind of nonlocal equations. 
%With the use of the change of variable, the equation \eqref{P} is reduced to a ``local'' semilinear equation,
%for which various tools are available to solve it. 
%
%For other type of change of variables in this type of problems, see also \cite{ACM,AA}.

%\begin{remark}
% We believe that change of variable of this type can be used also to deal with other nonlocal equations.
%\end{remark}

The paper is organized as follows. 

In Section \ref{sec:prelim} we 
present the general approach to solve problem \eqref{P}
and give a useful lemma (see Lemma \ref{le}).
In Section \ref{sec:f(x)} we prove our  result in the case $f=f(x)$, i.e. Theorem \ref{th:main1}.
In Section \ref{sec:general} we consider the general case $f=f(x,u)$ proving Theorem \ref{th:main2}.

\medskip

In all the paper we denote with $W^{m,p}(\Omega)$ the usual Sobolev spaces. Whenever $p=2$
we use the notation $H^{m}(\Omega)$. Finally $H^{1}_{0}(\Omega)$ is the closure of the test functions
with respect to the norm in $H^{1}(\Omega)$.

\section{Preliminaries}\label{sec:prelim}

We attack problem \eqref{P}, in both cases $f=f(x)$ and $f=f(x,u)$, in the following way.
Firstly, for every fixed $r\geq0$, we consider the auxiliary problem
\begin{equation*}%\label{Pr}\tag{$P_{r}$}
\left \{ \begin{array}{ll}
\Delta^{2}u-\text{div}\left(m_{r}\left (u\right)\nabla u\right)= f(x,u) & \mbox{in $\Omega$,}\\
u=\Delta u=0 & \mbox{on $\partial\Omega$,}
\end{array}\right.
\end{equation*}
associated to \eqref{P} for which we prove the existence of a unique solution $u_{r}$.
Secondly, we show that the map $$S: r\mapsto \int_{\Omega}|\nabla u_{r}|^{2}dx$$ 
has a fixed point $r_{*}$, which of course gives a solution $ u_{*} := u_{r_{*}}$
of the original problem \eqref{P}.

\medskip

\medskip

Let us define the map $M:(t, r)\in\R\times[0,+\infty)\mapsto \int_{0}^{t}m(s, r)ds\in  \R$; 
it will be convenient  also to introduce the notation, for every $r\geq0$:
$$M_{r}:=M(\cdot, r):\R\to \R.$$
In the next lemma we list some properties of $M$ which will have an important role in the study of problem \eqref{P}
in the next two sections.

\begin{lemma}\label{le}
Assume \eqref{m_{0}}-\eqref{m_{1}}. Then, \smallskip
\begin{enumerate}
  \item[$(a)$]  for each $r\in [0,\infty)$ the map $M_{r}:\R \to\R $ is a strictly increasing diffeomorphism
  (in particular it satisfies the ``sign condition'' $M_{r}(t)t\geq 0$ for all $t\in \R$) and moreover $|M_{r}(t)|>\mathfrak{m}|t|$
  for every $t\neq 0$.
  \medskip
   
   \item[$(b)$] The map $M$ is continuous. \medskip
     
     \item[$(c)$] For each $r\in [0,\infty),$ the inverse map $ M^{-1}_{r}$  is  Lipschitz continuous  with Lipschitz constant $\mathfrak m^{-1}$.
     In particular $|M_{r}^{-1}(t)|\leq \mathfrak m^{-1}|t|$ for all $t\in\R$.\medskip
%  \item[$(b)$] if $t_{1}, t_{2}\in \R$ and $r\geq 0$, $(M_{r}(t_{1})-M_{r}(t_{2}))(t_{1}-t_{2})\geq 0$. Moreover, for each
%$t\neq 0$, $|M_{r}(t)|>\mathfrak{m}|t|$. 
  % $t_{n}\to t_{0}$ and $r_{n}\to r_{0}$ then $M_{r_{n}}(t_{n})\to M_{r_{0}}(t_{0})$;
\end{enumerate}

Assume now also that
\begin{enumerate}[label=(m\arabic*),ref=m\arabic*,start=2] 
\item\label{m_{2}} for each $r\in [0,+\infty)$ the map 
$m_{r}:\R\to(0,+\infty)$ is strictly decreasing in $(-\infty,0)$  and strictly increasing in $(0,+\infty)$. \medskip 
\end{enumerate}
Then
 \begin{enumerate}
  \item[$(d)$]   if $t_{n}\to t_{0}$ and $r_{n}\to r_{0}$ then $M_{r_{n}}^{-1}(t_{n})\to M_{r_{0}}^{-1}(t_{0})$.
  %In particular, for each fixed $s\in\R$, application $r\to M_{r}^{-1}(s)$ is continuous.   
 % For each fixed $s\in\R$, the map  $r\mapsto M_{r}^{-1}(s)$ is (continuous and) strictly decreasing; \medskip
%  \item[$e)$] for each fixed $s\in\R$, application $r\to M_{r}^{-1}(s)$ is decreasing.
 \medskip
      \item[$(e)$] For each $r\in[0,+\infty)$, the map $t\mapsto M_{r}^{-1}(t)/t$ is (continuous and) strictly decreasing in $(0,+\infty)$ and strictly increasing  in $(-\infty, 0)$. 
\end{enumerate}
\end{lemma}

%\begin{remark}\label{rem:comentario}
%Before giving the proof we observe the following.
%\begin{itemize}
%\item[(1)]
 In the remaining of the paper we will use just itens (a) and (b) of this Lemma.
We have preferred to include also condition \eqref{m_{2}} and itens (c), (d), (e)
just to have a complete list of properties which may be useful for future references.
%\item[(2)] Property \eqref{m_{2}} was fundamental in \cite{SS} in order to prove that the map
%$S$ defined above is continuous. Here since we will deal with a system, the situation appears better.
%\end{itemize}
%\end{remark}

\begin{proof}
(a) For the first part see \cite[Lemma 2.1]{SS}. 
%The sign condition follows then by the fact that $M_{r}(0)=0$.
Moreover, if $t<0$, from \eqref{m_{1}}, 
$$
|M_{r}(t)|=-\int_{0}^{t}m_{r}(s)ds=\int_{t}^{0}m_{r}(s)ds>-\mathfrak{m}t=\mathfrak{m}|t|.
$$
The argument  is similar if $t>0$.

\noindent (b) Let $t_{n}\to t_{0}$ and $r_{n}\to r_{0}$. If $t_{0}=0$, then
$$
|M_{r_{n}}(t_{n})|\leq |t_{n}|\max_{t\in [-\varepsilon, \varepsilon]\atop r\in [0, r_{0}+\varepsilon]}m(t, r),
$$
for all $n\geq n_{0}$ and some $\varepsilon>0$. Therefore,
$$
M_{r_{n}}(t_{n})\to 0=M_{r_{0}}(0).
$$

Suppose now $t_{0}>0$. Denote with $A_{n}=\min\{t_{0}, t_{n}\}$ and $B_{n}=\max\{t_{0}, t_{n}\}$. Note that there are $0<T_{1}<t_{0}<T_{2}$ such that if $n\geq n_{0}$ then $T_{1}\leq A_{n}\leq t_{0}\leq B_{n}\leq T_{2}$. Thus,
$$
|M_{r_{n}}(t_{n})-M_{r_{0}}(t_{0})|=\left|\int_{0}^{A_{n}}[m(t, r_{n})-m(t, r_{0})]dt-\int_{A_{n}}^{B_{n}}m(t, s_{n})dt\right|,
$$
where $s_{n}= r_{n}$ if $t_{n} \geq t_{0}$ and $s_{n}=r_{0}$ if $t_{0}>t_{n}.$
Consequently
\begin{equation*}
|M_{r_{n}}(t_{n})-M_{r_{0}}(t_{0})|\leq\int_{0}^{t_{0}}|m(t, r_{n})-m(t, r_{0})|\chi_{[0, A_{n}]}(t)dt+(B_{n}-A_{n})\max_{t\in[T_{1}, T_{2}]\atop s\in [T_{3},T_{4}]}m(t, s), 
\end{equation*}
for all 
$n\geq n_{0}$ where $\{s_{n}\}\subset [T_{3},T_{4}]$ and $\chi_{[0, A_{n}]}$ is the characteristic function of the interval
 $[0, A_{n}]$.

Since
$$
|m(t, r_{n})-m(t, r_{0})|\chi_{[0, A_{n}]}(t)\to 0, \ \forall \ t\in[0, t_{0}]
$$
and
$$
|m(t, r_{n})-m(t, r_{0})|\chi_{[0, A_{n}]}(t)\leq 2\max_{t\in [0, t_{0}]\atop r\in [0, r_{0}+\varepsilon]}m(t, r),
$$
for all $n\geq n_{0}$ and some $\varepsilon>0$, it follows by the Lebesgue Dominated Convergence Theorem that
$$
M_{r_{n}}(t_{n})\to M_{r_{0}}(t_{0}).
$$
The argument is similar for $t_{0}<0$.

The proof of (c)-(d)-(e) is given in \cite[Lemma 2.1]{SS}.
\end{proof}

%\textcolor{red}{Throughout this article we use $z$ and $w$ to denote, respectively, the first and the second variables of a couple $(., .)$.}

\medskip

\section{The particular case $f(x, u)=f(x)$}\label{sec:f(x)}
In this section we address  the problem
\begin{equation}\label{NP}%\tag{$P_{f(x)}$}
\left \{ \begin{array}{ll}
\Delta^{2}u-\text{div}\left(m\left (u, |\nabla u|_{2}^{2}\right)\nabla u\right)= f(x) & \mbox{in $\Omega$,}\\
u=\Delta u=0 & \mbox{on $\partial\Omega$,}
\end{array}\right.
\end{equation}
with $f\not\equiv0.$
%where $\Omega\subset\R^{N}$ is a bounded smooth domain, $f\in L^{q}(\Omega), q>N/2$ and  $N\geq 2$.
%
%\medskip
%We prove the following
%
%\textcolor{red}{TOGLIERE}
%\begin{theorem}
%If \eqref{m_{0}}-\eqref{m_{2}} hold, $0\not\equiv f\in L^{q}(\Omega)$ and $q>N/2$,  
%then problem \eqref{NP} has a nontrivial weak solution $u_{*}$.
%\end{theorem}
%
In order to prove Theorem \ref{th:main1}, let us consider for every $r\geq0$ the auxiliary problem
\begin{equation}\label{Pr}\tag{$P_{r}$}
\left \{ \begin{array}{ll}
\Delta^{2}u-\text{div}\left(m_{r}(u)\nabla u\right)= f(x) & \mbox{in $\Omega$,}\\
u=\Delta u=0 & \mbox{on $\partial\Omega$,}
\end{array}\right.
\end{equation}
 whose weak solution is, by definition a function $u_{r}$ such that
\begin{equation}\label{eq:Pr1}
u_{r}\in H^{2}(\Omega)\cap H_{0}^{1}(\Omega)\cap L^{\infty}(\Omega)
\end{equation} 
and
\begin{equation}\label{eq:Pr2}
\int_{\Omega}\Delta u_{r}\Delta \varphi dx+\int_{\Omega}m_{r}(u_{r})\nabla u_{r}\nabla \varphi dx=\int_{\Omega}f(x)\varphi dx, \ \forall 
\varphi\in H^{2}(\Omega)\cap H_{0}^{1}(\Omega).
\end{equation}
By our assumptions the equality above makes sense. 

\medskip

Let us consider the following system (recall that $M_{r}(t) = \int_{0}^{t} m_{r}(s)ds$):
\begin{equation}\label{star}\tag{**}
\left \{ \begin{array}{ll}
-\Delta( M_{r}(u)-v)= f(x) & \mbox{in $\Omega$,}\\
-\Delta u=-v & \mbox{in $\Omega$,}\\
u=v=0 & \mbox{on $\partial\Omega$,}
\end{array}\right.
\end{equation}
whose solution is by definition a pair $(u_{r}, v_{r})$ such that
\begin{equation}\label{star1}
u_{r}\in H^{1}_{0}(\Omega)\cap L^{\infty}(\Omega),\quad  M_{r}(u_{r}), v_{r}\in H^{1}_{0}(\Omega)
\end{equation}
and
\begin{equation}\label{star2}
\int_{\Omega} \nabla (M_{r}(u_{r})-v_{r}) \nabla \xi dx=\int_{\Omega} f(x)\xi dx,\quad 
\int_{\Omega}\nabla u_{r}\nabla \xi dx =- \int_{\Omega}v_{r}\xi dx, \quad \forall \xi\in H^{1}_{0}(\Omega).
\end{equation}

The equivalence between \eqref{Pr} and system \eqref{star} is given in the next

\begin{proposition}
If $u_{r}$ is a weak solution of \eqref{Pr} which is also in $W^{3,2}(\Omega)$, then the pair
$(u_{r}, \Delta u_{r})$ is a weak solution of \eqref{star}.

If a pair $(u_{r}, v_{r})$ solves \eqref{star} and $u_{r}\in W^{3,2}(\Omega)$, then necessarily
$v_{r}=\Delta u_{r}$
and $u_{r}$ is a solution of \eqref{Pr}.
\end{proposition}
\begin{proof}
If $u_{r}$ is a weak solution of \eqref{Pr} which is also in $W^{3,2}(\Omega)$, then $v_{r}=\Delta u_{r}\in H^{1}_{0}
(\Omega)$
and the second identity in \eqref{star2} is trivially satisfied. On the other hand, since
$\Delta (M_{r}(u_{r}))=\text{div}(m_{r}(u_{r})\nabla u_{r})$ in a weak sense, 
the first identity in \eqref{star2} is satisfied (with $v_{r}=\Delta u_{r}$) since it is just a consequence of \eqref{eq:Pr2}.
It remains to show that $M_{r}\in H^{1}_{0}(\Omega)$. However, since $M_{r}$ is continuous, $u_{r}\in L^{\infty}(\Omega)$ and $\Omega$ is bounded, it is trivially $M_{r}(u_{r})\in L^{2}(\Omega)$. Moreover,
$$\int_{\Omega} |\nabla M_{r}(u_{r})|^{2}dx =\int_{\Omega} |m_{r}(u_{r})\nabla u_{r}|^{2}dx<\infty$$
showing that $M_{r}(u_{r})\in H^{1}_{0}(\Omega)$ and then \eqref{star1}.

On the contrary assume that a pair $(u_{r},v_{r})$ is a solution of \eqref{star},
and $u_{r}\in W^{3,2}(\Omega)$. In particular $u_{r}\in H^{2}(\Omega)$, which proves \eqref{eq:Pr1}.
Moreover (by definition of solution) $v_{r}\in H^{1}_{0}(\Omega)$
and from the second equation in \eqref{star2} we get $\Delta u_{r}=v_{r}\in H^{1}_{0}(\Omega)$.
Since again
$\Delta (M_{r}(u_{r}))=\text{div}(m_{r}(u_{r})\nabla u_{r})$ in a weak sense, 
we have that \eqref{eq:Pr2} is satisfied since is a consequence
of the first identity in \eqref{star1}.
\end{proof}

Moreover with a further change of variable 
\begin{equation}\label{eq:change}
z:=u, \qquad w:=M_{r}(u) - v,
\end{equation}
 system \eqref{star} can be written as
\begin{equation}\label{System}\tag{$S_{r}$}
\left \{ \begin{array}{ll}
-\Delta w= f(x) & \mbox{in $\Omega$,}\\
-\Delta z+M_{r}(z)= w & \mbox{in $\Omega$,}\\
z=w=0 & \mbox{on $\partial\Omega$.}
\end{array}\right.
\end{equation}
%\begin{equation}\label{star}\tag{*}
%\left \{
%\begin{array}{ll}
%-\Delta( M_{r}(u)-v)= f(x) & \mbox{in $\Omega$,} \\
%-\Delta u=-v & \mbox{in $\Omega$,} \\
%u=v=0 & \mbox{on $\partial\Omega$,}
%\end{array}
%\right.
%\end{equation}
%which by definition means that $u_{r}\in H^{1}_{0}(\Omega)\cap L^{\infty}(\Omega)$ with
%$M_{r}(u_{r})\in H^{1}_{0}(\Omega), v_{r}\in H^{1}_{0}(\Omega)$
%and   for all $\xi\in H^{1}_{0}(\Omega)$ we have
%\begin{eqnarray*}
%\int_{\Omega} \nabla (M_{r}(u_{r}) - v_{r})\nabla \xi  dx& =&\int_{\Omega}f(x)\xi dx \\
%\int_{\Omega} \nabla u_{r} \nabla \xi dx &=& \int_{\Omega} v_{r}\xi dx.
%\end{eqnarray*}
%On the other hand, if we have a pair $(u_{r}, v_{r})$ solution of \eqref{starstar}, this means that
%$u_{r}\in H^{1}_{0}(\Omega)\cap L^{\infty}(\Omega)$
Clearly $(u_{r},v_{r})\in \left(H^{1}_{0}\cap W^{3,2}(\Omega) \cap L^{\infty}(\Omega)\right)\times H^{1}_{0}(\Omega)$ is a solution of \eqref{star}, in the sense that the identities in \eqref{star2} hold,
 if and only if the pair $(z_{r}, w_{r}):=(u_{r},M_{r}(u_{r})- v_{r})$ solves \eqref{System}
 with additionally $u_{r}\in W^{3,2}(\Omega)$. Observe now that actually $w_{r}$ does not depend on $r$.

\medskip

In order to prove the existence of solution to problem \eqref{System}, and then \eqref{Pr},
let us recall first the following result concerning the general Dirichlet problem 
\begin{equation*}
\left \{ \begin{array}{ll}
-\Delta u+g(u)= \mu(x) & \mbox{in $\Omega$,}\\
u=0 & \mbox{on $\partial\Omega$}
\end{array}\right.
\end{equation*}
whose associated energy  functional we denote with $E$.
\begin{lemma}(See \cite[Lemma 2.4]{P1})\label{le:Ponce}
Let $g :\R \to \R $ be a continuous function satisfying the sign condition  $g(t)t\geq0$
and let 
$k \in\R$. Given $\mu\in L^{\infty}(\Omega)$ let $u \in H_{0}^{1}(\Omega)$ be a minimizer of
the functional $E$ on $H^{1}_{0}(\Omega)$.
\begin{itemize}
\item[i)] If for every $t\geq k$, $g(t)\geq|\mu|_{\infty}$ then $u\leq k$ in $\Omega$.
\item[ii)] If for every $t\leq k$, $g(t)\leq - |\mu|_{\infty}$ then $u\geq k$ in $\Omega$.
\end{itemize}
\end{lemma}
With this result in hands we can prove the following.
\begin{proposition}\label{exis}
If \eqref{m_{0}}-\eqref{m_{1}} hold and $0\not\equiv f\in L^{q}(\Omega), q>N/2$, 
then for each $r\geq 0$, the auxiliary problem \eqref{System} has a unique nontrivial weak solution
 $(z_{r}, w)$.
Even more, it is $z_{r}\in H_{0}^{1}(\Omega)\cap W^{3,2}(\Omega)\cap C^{1, \alpha}(\overline{\Omega})$, for some $\alpha\in (0, 1)$.
\end{proposition}

\begin{proof}
%By the previous considerations
% it is sufficient to prove that, for any $r\geq 0$, system \eqref{System} has a unique solution 
%$(u_{r}, w_{r})\in \left(H_{0}^{1}(\Omega)\cap W^{3,2}(\Omega)\cap L^{\infty}(\Omega)\right)\times H^{1}_{0}(\Omega)$
%and that it is also $u_{r}\in C^{1,\alpha}(\overline \Omega)$.
%
Let $r\geq0$ be fixed and let us proceed in two steps.
\medskip

{\bf Step 1: Existence}

\medskip

It is clear that  the
first equation in \eqref{System} has a unique nontrivial weak solution $w\in H^{1}_{0}(\Omega)\cap 
C(\overline{\Omega})$, being  $q>N/2$, which obviously does not depend on $r$.
% On the other hand, from \eqref{m_{0}} we conclude that $M_{r}$ is increasing. 
%\textcolor{red}{TOGLIERE: Since $M_{r}(0)=0$, it follows that $M_{r}$ satisfies the sign condition
%\begin{equation}\label{sign}
%M_{r}(t)t\geq 0, \ \forall t\in \R.
%\end{equation}}

By the sign condition of $M_{r}$ in Lemma \ref{le} (a) we conclude the existence of numbers $\tau_{1}<0<\tau_{2}$ such that $M_{r}(\tau_{1})=-|w|_{\infty}$ and $M_{r}(\tau_{2})=|w|_{\infty}$. %Since $M_{r}$ is increasing 
Moreover we can invoke Lemma \ref{le:Ponce} 
 ensuring that if $z_{r}\in H_{0}^{1}(\Omega)$ is a minimum of the functional $I_{r}:H_{0}^{1}(\Omega)\to(\infty, \infty]$ associated to the second equation in \eqref{System}, given explicitly by
$$
I_{r}(z)=\frac{1}{2}\int_{\Omega}|\nabla z|^{2} dx+\int_{\Omega}\widehat{M}_{r}(z) dx-\int_{\Omega}wz dx,
$$ 
then $\tau_{1}\leq z_{r}\leq \tau_{2}$ and, therefore, $z_{r}\in L^{\infty}(\Omega)$. Above, $\widehat{M_{r}}$ denotes the primitive of $M_{r}$ with regard to $t$, such that $\widehat{ M_{r}}(0)=0$.

We are then reduced to find a minimum of $I_{r}$ which it turns out to be unique.
 The problem with this functional is  
that the integral $\int_{\Omega}\widehat{M_{r}}(z) dx$, $z\in H^{1}_{0}(\Omega)$ can be infinite, then
our strategy is to  truncate the functional  $I_{r}$ in such a way that now,
 by standard methods, there is a unique critical point  $z_{r}$ of the truncated functional
 which is also the unique critical point
of $I_{r}$ and it belongs to $H^{1}_{0}(\Omega)\cap W^{3,2}(\Omega)\cap L^{\infty}(\Omega)$.
Then we will prove that $z_{r}\in C^{1,\alpha}(\overline\Omega)$.
% to prove the existence of weak solutions we do not differentiate directly the functional $I_{r}$; indeed
%   we are going to define a truncated functional $I_{r,\tau}\in C^{1}(H^{1}_{0}(\Omega), \R)$ whose critical point is the unique weak solution of the second equation in \eqref{System}.

So let us consider the truncated map
\begin{equation*}\label{trunc}
M_{r, [\tau_{1},\tau_{2}]}(t)=\left \{ \begin{array}{ll}
M_{r}(\tau_{1}) & \mbox{if $t\leq \tau_{1}$,}\\
M_{r}(t) & \mbox{if $\tau_{1}< t< \tau_{2}$,}\\
M_{r}(\tau_{2}) & \mbox{if $\tau_{2}\leq t$}
\end{array}\right.
\end{equation*}
with $\tau_{1},\tau_{2}$ given above, and the new functional 
$$
I_{r, [\tau_{1}, \tau_{2}]}(z)=\frac{1}{2}\int_{\Omega}|\nabla z|^{2} dx+\int_{\Omega}\widehat{M}_{r, [\tau_{1},\tau_{2}]}(z) dx-\int_{\Omega}wz dx,
$$
which is  well defined (because $M_{r, [\tau_{1},\tau_{2}]}$ is a bounded function) and belongs to $C^{1}(H^{1}_{0}(\Omega), \R)$. 
Again  $\widehat{M}_{r, [\tau_{1},\tau_{2}]}(t)=\int_{0}^{t}M_{r, [\tau_{1},\tau_{2}]}(s) ds$. 

%\textcolor{red}{Our aim is to show nw that the critical points of $I_{r, \tau}$ are weak solutions of the second equation in  \eqref{System}.}
%
%\textcolor{red}{Since $M_{r, [\tau_{1},\tau_{2}]}$ is continuous and verifies the sign condition in Lemma \ref{le} (a), we can ensure 
%by Lemma \ref{le:Ponce}
%the existence of a nontrivial global minimum  $u_{r}$ of $I_{r, [\tau_{1},\tau_{2}]}(z)$.??}
Note that   the functional 
$I_{r, [\tau_{1},\tau_{2}]}$ is  coercive. Moreover
by \eqref{m_{1}} the function  $\widehat{M}_{r, [\tau_{1},\tau_{2}]}$ is convex, and then the functional 
$I_{r, [\tau_{1},\tau_{2}]}$ is also strictly convex and   weakly lower semicontinuous. Then
 it possesses  a unique  critical point, which is minimum, $z_{r}\in H^{1}_{0}(\Omega)$.
Thus,
$$
\int_{\Omega}\nabla u_{r}\nabla \varphi dx+\int_{\Omega}\widehat M_{r,[\tau_{1},\tau_{2}]}(u_{r})\varphi dx=\int_{\Omega} w\varphi dx, \ \forall  \varphi\in H^{1}_{0}(\Omega).
$$ 
As $M_{r,  [\tau_{1},\tau_{2}]}$ is nondecreasing and satisfies the conditions of Lemma \ref{le:Ponce}, 
it holds $\tau_{1}\leq z_{r}\leq \tau_{2}$. Therefore $z_{r}\in H_{0}^{1}(\Omega)\cap L^{\infty}(\Omega)$ is a nontrivial minimum point of the original functional $I_{r}$.
%, and then a weak solution of the second equation in  \eqref{System}.

%Note that since by \eqref{m_{0}} the function  $\widehat{M}_{r, \tau}$ is convex, the functional $I_{r, \tau}$ is strictly convex and, consequently, $u_{r}$ is its unique critical point.

Now we prove that it is also $z_{r}\in  W^{3,2}(\Omega)\cap C^{1, \alpha}(\overline{\Omega})$, for some $\alpha\in (0, 1)$.
In fact, by defining the  function 
$$
g(x)=\begin{cases}
M_{r}(z_{r}(x))/z_{r}(x) &\text{ if } x\in \{x\in \Omega: z_{r}(x)\neq 0\}\\
0 &\text{ otherwise}
\end{cases}
$$
we have
$g\in L^{\infty}(\Omega)$, $g\geq 0$.
%by $g(x)=M_{r}(u_{r})/u_{r}$ in $\{x\in \Omega: u(x)\neq 0\}$ and $g(x)=0$ in $\{x\in \Omega: u(x)=0\}$, 
So $z_{r}$ is a weak solution of the problem
\begin{equation*}\label{problem1}
\left \{ \begin{array}{ll}
-\Delta z+ g(x)z=w & \mbox{in $\Omega$,}\\
z=0 & \mbox{on $\partial\Omega$}
\end{array}\right.
\end{equation*}
and recalling that $w\in C(\overline{\Omega})$, it follows from \cite[Theorem 9.15]{GT} that $z_{r}\in W_{0}^{2, p}(\Omega)$ for all $p\in (1, \infty)$. By the Sobolev embedding we get $z_{r}\in C^{1, \alpha}(\overline{\Omega})$ for some $\alpha\in(0, 1)$. 
Finally, being $w\in H_{0}^{1}(\Omega)$ and $z_{r}\in H_{0}^{1}(\Omega)\cap L^{\infty}(\Omega)$ we get $w-M_{r}(z_{r})\in H_{0}^{1}(\Omega)$: therefore from 
\begin{equation*}\label{problem2}
\left \{ \begin{array}{ll}
-\Delta z_{r}=w-M_{r}(z_{r}) & \mbox{in $\Omega$,}\\
z_{r}=0 & \mbox{on $\partial\Omega$,}
\end{array}\right.
\end{equation*}
we conclude that $z_{r}\in  W^{3,2}(\Omega)$,
see \cite{ADN}.

\medskip

{\bf Step 2: Unicity}

\medskip

We show that if $\overline z_{r}\in H_{0}^{1}(\Omega)\cap L^{\infty}(\Omega)$ is another weak solution of 
the second equation in \eqref{System} then $\overline z_{r}=z_{r}$. In fact, suppose $|\overline z_{r}|_{\infty}<c$. If $c\leq \min\{-\tau_{1}, 
\tau_{2}\}$ there is nothing to do. On the other hand, if $c>\min\{-\tau_{1}, \tau_{2}\}$, 
we set $\widehat{c}=\max\{c, \max\{-\tau_{1}, \tau_{2}\}\}$ and then there
 exists numbers $\tau_{1}'<0<\tau_{2}'$ such that $M_{r}(\tau_{1}') = -\widehat c, M_{r}(\tau_{2}') = \widehat c $.
 As before we define the truncated map
 \begin{equation*}\label{trunc}
M_{r, [\tau'_{1},\tau'_{2}]}(t)=\left \{ \begin{array}{ll}
M_{r}(\tau'_{1}) & \mbox{if $t\leq \tau'_{1}$,}\\
M_{r}(t) & \mbox{if $\tau'_{1}< t< \tau'_{2}$,}\\
M_{r}(\tau'_{2}) & \mbox{if $\tau'_{2}\leq t$}
\end{array}\right.
\end{equation*}
and the corresponding functional $I_{r,[\tau'_{1},\tau'_{2}]}$.
%
%and we
%we define the truncated map 
%$M_{r, \overline{c}}$, as in \eqref{trunc}, where $\overline{c}=\max\{c, \max\{-\tau_{1}, \tau_{2}\}\}$, with $\|w\|
%_{C(\overline{\Omega})}, \tau_{1}$ and $\tau_{2}$ replaced by $\overline{c}, \tau_{1}'$ and $\tau_{2}'$, respectively, 
%i.e., $M_{r}(\tau_{1}')=-\overline{c}$ and $M_{r}(\tau_{2}')=\overline{c}$. Let $I_{r, \overline{c}}$ be the correspondent 
%functional. 
Using the previous arguments, we conclude that $I_{r, [\tau'_{1},\tau'_{2}]}$ has a unique critical point which is a global 
minimum. Since $\overline z_{r}$ and $z_{r}$ are critical points of the strictly convex functional 
$I_{r, [\tau'_{1},\tau'_{2}]}$, the result follows.
\end{proof}

\begin{remark}\label{rem1}
It follows by the previous proof that 
\begin{equation}\label{ineq1}
|M_{r}(z_{r})|_{\infty}\leq |w|_{\infty}, \ \forall\,  r\geq 0
\end{equation}
and so,  by Lemma \ref{le} (a), the uniform estimate in $r$:% and last inequality we get the following estimate to solution $u_{r}$:
\begin{equation*}\label{priori}
|z_{r}|_{\infty}< \frac{1}{\mathfrak{m}}|w|_{\infty}, \ \forall\,  r\geq 0.
\end{equation*}
This will be fundamental in the case of the general nonlinearity $f(x,u)$.
\end{remark}

As a consequence, recalling that actually $z_{r} = u_{r}$ (see the change of variable \eqref{eq:change})
we have 
 
\begin{proposition}\label{prop:ur}
If \eqref{m_{0}}-\eqref{m_{1}} hold and $0\not\equiv f\in L^{q}(\Omega), q>N/2$, 
 for each $r\geq 0$, problem \eqref{Pr} admits a unique nontrivial solution $u_{r}\in H_{0}^{1}(\Omega)\cap W^{3,2}(\Omega)\cap C^{1, \alpha}(\overline{\Omega})$, for some $\alpha\in (0, 1)$. 
\end{proposition}

\begin{remark}\label{rem2}
We stress here  that if $u_{r}\in H_{0}^{1}(\Omega)\cap L^{\infty}(\Omega)$ is the solution of \eqref{Pr} then $w:=M_{r}(u_{r})-\Delta u_{r}$ does not depend on $r$, due to uniqueness of the solution to the first equation 
appearing in \eqref{System}.
% \textcolor{red}{por que? (see Proposition \ref{equiv})}. 
\end{remark}

The next step is given by the following proposition.

\begin{proposition}\label{prop}
If \eqref{m_{0}}-\eqref{m_{1}} hold and $u_{r}\in H_{0}^{1}(\Omega)\cap W^{3,2}(\Omega)\cap C^{1, \alpha}(\overline{\Omega})$ is the unique solution of \eqref{Pr} provided in Proposition \ref{prop:ur}, then the map
 $$S: r\in[0,+\infty)\longmapsto \int_{\Omega}|\nabla u_{r}|^{2}dx\in \mathbb [0,+\infty)$$ is continuous, bounded
 and has a fixed point $r_{*}$.
\end{proposition}

\begin{proof}
Let $\{r_{n}\}$ be a sequence of nonnegative numbers such that $r_{n}\to r_{\infty}\geq0$.
Setting $$u_{n}:=u_{r_{n}}\in H_{0}^{1}(\Omega)\cap W^{3,2}(\Omega)\cap C^{1, \alpha}(\overline{\Omega}),\qquad 
M_{n}:=M_{r_{n}}, \quad n\in \mathbb N$$ and using the second equation in \eqref{System}, we have
$$
\int_{\Omega}|\nabla u_{n}|^{2}dx+\int_{\Omega}M_{n}(u_{n})u_{n} dx=\int_{\Omega}wu_{n} dx, \ \forall\, n\in\N.
$$
Hence, from the sign condition, the H\"older  and Poincar\'e inequalities we infer
\begin{equation}\label{end}
\int_{\Omega}|\nabla u_{n}|^{2}dx\leq \frac{1}{\sqrt{\lambda_{1}}}|w|_{2}\left( \int_{\Omega}|\nabla u_{n}|^{2}dx\right)^{1/2}, \ \forall \,n\in\N
\end{equation}
%Whence $\left( \int_{\Omega}|\nabla u_{n}|^{2}dx\right)^{1/2}\leq (1/\sqrt{\lambda_{1}})|w|_{2}$ and 
so that $\{u_{n}\}$ is bounded in $H_{0}^{1}(\Omega)$. Therefore there exists $u_{\#}\in H_{0}^{1}(\Omega)$ such that, passing to a subsequence, 
\begin{equation}\label{conv2}
u_{n}\rightharpoonup u_{\#} \ \mbox{in} \ H_{0}^{1}(\Omega), \  u_{n}\to u_{\#} \ \mbox{in} \ L^{s}(\Omega), \ 1\leq s\leq 2^{\ast}, \ u_{n}(x)\to u_{\#}(x) \ \mbox{a.e. in $\Omega$}
\end{equation}
%\begin{equation*}
%u_{n}\to u_{\ast} \ \mbox{in} \ L^{s}(\Omega), \ 1\leq s\leq 2^{\ast},
%\end{equation*}
%\begin{equation}\label{conv1}
%u_{n}(x)\to u_{\ast}(x) \ \mbox{a.e. in $\Omega$}
%\end{equation}
and
%\begin{equation}
$|u_{n}(x)|\leq h(x) \ \mbox{a.e. in $\Omega$},$
%\end{equation}
for some $h\in L^{2}(\Omega)$.

We will show now that $u_{\#}=u_{\infty}$, where $u_{\infty}:=u_{r_{\infty}}$. In fact, from the definition of $u_{n}$
$$
\int_{\Omega}\nabla u_{n}\nabla \varphi dx+\int_{\Omega}M_{n}(u_{n})\varphi dx=\int_{\Omega}w\varphi dx, \ \forall\, \varphi\in H_{0}^{1}(\Omega).
$$
Passing to the limit in $n\to\infty$ and using \eqref{ineq1}, \eqref{conv2} %\eqref{conv1} 
 and the Lebesgue Dominated Convergence Theorem, we get (hereafter $M_{\infty}:=M_{r_{\infty}}$)
\begin{equation}\label{equal.}
\int_{\Omega}\nabla u_{\#}\nabla \varphi dx+\int_{\Omega}M_{\infty}(u_{\#})\varphi dx=\int_{\Omega}w\varphi dx, \ \forall\, \varphi\in H_{0}^{1}(\Omega).
\end{equation}
Equality \eqref{equal.} implies $u_{\#}=u_{\infty}$.

Using again the second equation in \eqref{System} we get
\begin{eqnarray*}
\int_{\Omega}|\nabla u_{n}-\nabla u_{\infty}|^{2}dx&=&-\int_{\Omega}[M_{n}(u_{n})-M_{\infty}(u_{\infty})](u_{n}-u_{\infty})dx\\
&\leq&\frac{1}{\sqrt{\lambda_{1}}}|M_{n}(u_{n})-M_{\infty}(u_{\infty})|_{2}\left(\int_{\Omega}|\nabla u_{n}-\nabla u_{\infty}|^{2}dx\right)^{1/2}, \ \forall\,n\in\N.
\end{eqnarray*}
Thus,
\begin{equation}\label{ineq2}
\left(\int_{\Omega}|\nabla u_{n}-\nabla u_{\infty}|^{2}dx\right)^{1/2}\leq \frac{1}{\sqrt{\lambda_{1}}}|M_{n}(u_{n})-M_{\infty}(u_{\infty})|_{2}.
\end{equation}
From \eqref{conv2} and the continuity of $M$ (see Lemma \ref{le} (b)), we deduce
$$
|M_{n}(u_{n}(x))-M_{\infty}(u_{\infty}(x))|\to 0 \ \mbox{a.e. in $\Omega$}
$$
and, from \eqref{ineq1},
$$
|M_{n}(u_{n}(x))-M_{\infty}(u_{\infty}(x))|\leq 2|w|_{\infty}.
$$
So we can  use the Lebesgue Dominated Convergence Theorem in  \eqref{ineq2} and  deduce that
$$
S(r_{n})=\int_{\Omega}|\nabla u_{n}|^{2}dx\to \int_{\Omega}|\nabla u_{\infty}|^{2}dx=S(r_{\infty}),
$$
i.e. the map $S$ is continuous.
%\end{proof}
%\subsection{Proof  of Theorem \ref{th:main1}}
%As we already said, the idea is to show that 
%the continuous map  $S$ defined in Proposition \ref{prop}
%  has a fixed point. 

Clearly, being $u_{0}$ a nontrivial solution of (\ref{Pr}) with $r=0$, it holds
\begin{equation*}\label{9}
S(0)=\int_{\Omega}|\nabla u_{0}|^{2}dx>0.
\end{equation*}
On the other hand, by the definition of $u_{r}$, we can argue as in \eqref{end} to conclude that
$$
S(r)\leq\frac{1}{\lambda_{1}}|w|_{2}^{2}<r,
$$
for all $r$ large enough. We deduce that $S$ has a fixed point $r_{*}>0$.
%and then, for $u_{*} := u_{r_{*}}$, we have
%$$
%r_{\ast}=S(r_{\ast})=\int_{\Omega}|\nabla u_{\ast}|^{2}dx.
%$$
%In other words,  $u_{*}$ is  a solution of  \eqref{NP}.
\end{proof}
Of course to $r_{*}$ is associated $ u_{*} := u_{r_{*}}$ which is a solution 
of the problem \eqref{NP}, proving Theorem \ref{th:main1}.

%\begin{remark}
%\textcolor{red}{Observe that we have showed that the map $S$ is bounded.}
%\end{remark}

%%%%%%%%%%%%%%%%%%%%%%%%%%%%%%%%%%%%%%%%%%%%%%%%%
%%%%%%%%%%%%%%%%%%%%%%%%%%%%%%%%%%%%%%%%%%%%%%%%%
%%%%%%%%%%%%%%%%%%%%%%%%%%%%%%%%%%%%%%%%%%%%%%%%%
%%%%%%%%%%%%%%%%%%%%%%%%%%%%%%%%%%%%%%%%%%%%%%%%%
%%%%%%%%%%%%%%%%%%%%%%%%%%%%%%%%%%%%%%%%%%%%%%%%%
%%%%%%%%%%%%%%%%%%%%%%%%%%%%%%%%%%%%%%%%%%%%%%%%%
%%%%%%%%%%%%%%%%%%%%%%%%%%%%%%%%%%%%%%%%%%%%%%%%%
%%%%%%%%%%%%%%%%%%%%%%%%%%%%%%%%%%%%%%%%%%%%%%%%%

\section{The general case $f(x,u)$}\label{sec:general}
In this Section we consider the problem
\begin{equation}\label{general}
\left \{ \begin{array}{ll}
\Delta^{2}u-\text{div}\left(m\left (u, |\nabla u|_{2}^{2}\right)\nabla u\right)= f(x, u) & \mbox{in $\Omega$,}\\
u=\Delta u=0 & \mbox{on $\partial\Omega$.}
\end{array}\right.
\end{equation}
and again we start by looking for  a solution $u_{r}$ of the auxiliary problem
(for every  fixed $r\geq0$)
\begin{equation}\label{P'r}\tag{$P'_{r}$}
\left \{ \begin{array}{ll}
\Delta^{2}u-\text{div}\left(m_{r}(u)\nabla u\right)= f(x, u) & \mbox{in $\Omega$,}\\
u=\Delta u=0 & \mbox{on $\partial\Omega$}
\end{array}\right.
\end{equation}
associated. Then we show that the map $$S: r\longmapsto \int_{\Omega}|\nabla u_{r}|^{2}dx$$ 
 has a fixed point, which will give of course a solution of the original problem.

%Hereafter $\lambda_{1}$ is the first eigenvalue of $-\Delta$ in $H^{1}_{0}(\Omega)$ and $\gamma$ denotes a positive constant (independent of $h\in L^{\infty}(\Omega)$ and $u\in C^{1,\alpha}(\Omega)$, $0<\alpha<1$) such that
%$$
%\|u\|_{C^{1, \alpha}(\Omega)}\leq \gamma |h|_{\infty},
%$$
%where
%\begin{equation}
%\left \{ \begin{array}{ll}
%-\Delta u= g(x) & \mbox{in $\Omega$,}\\
%u=0 & \mbox{on $\partial\Omega$.}
%\end{array}\right.
%\end{equation}
%
%In this section we are considering same assumptions \eqref{m_{0}}-\eqref{m_{2}} on $m$, moreover, $f:\Omega\times \R\to\R$ will be a Carath\'eodory function satisfying:
%\begin{enumerate}[label=(f\arabic*),ref=f\arabic*,start=7]
%\item \label{f_1}$f(x, 0)\neq 0$, \smallskip
% \item\label{f_2} there exists $\mu\in L^{\infty}(\Omega)$ and $\nu\in (0, \min\{\lambda_{1}^{(1+3\delta)/2}/|\Omega|^{1-\delta}, \mathfrak m^{\delta}/\gamma\})$, such that 
%  $$ 
%  |f(x, t)|\leq \mu(x) + \nu |t|^{\delta }\quad \mbox{ a.e. in $\Omega$ and $\forall t\in \R$},
%  $$
%  where $\delta\in (0,1]$.
%  \item\label{f_3} there is $\theta\in (0, \lambda_{1}^{2})$ such
%  $$
%  |f(x, t_{1})-f(x, t_{2})|\leq \theta|t_{1}-t_{2}| \quad \mbox{a.e. in } \Omega,
%  $$
% for all $t_{1}, t_{2}\in \R$. 
% 
%    \end{enumerate}

%\begin{remark}\label{110}

\medskip

As before, after reducing problem \eqref{P'r} in the unique unknown $u$
to a system in the two unknowns $u,v$ and considering the same change of variable $z:=u, w:=M_{r}(u)-v$ made in  Section \ref{sec:f(x)}, we see that $u_{r}\in H_{0}^{1}(\Omega)\cap W^{3,2}(\Omega)\cap L^{\infty}(\Omega)$ solves \eqref{P'r} if, and only if, the pair $( z_{r},  w_{r})=(u_{r}, M_{r}(u_{r})-v_{r})
\in \left(H_{0}^{1}(\Omega)
\cap W^{3,2}(\Omega)\cap L^{\infty}(\Omega)\right)\times H_{0}^{1}(\Omega)$ solves the following system 
\begin{equation}\label{Sys}\tag{$S'_{r}$}
\left \{ \begin{array}{ll}
-\Delta w= f(x, z) & \mbox{in $\Omega$,}\\
-\Delta z+M_{r}(z)= w & \mbox{in $\Omega$,}\\
z=w=0 & \mbox{on $\partial\Omega$,}
\end{array}\right.
\end{equation}
%\end{remark}
%A weak solution of \eqref{Sys} is a pair $(\overline z, \overline  w)\in (H_{0}^{1}(\Omega)\cap L^{\infty}(\Omega))\times %H_{0}^{1}(\Omega)$ such that
that is
\begin{multline*}
  \int_{\Omega}\nabla  z_{r}\nabla \varphi_{1} dx+\int_{\Omega}M_{r}
 ( z_{r})\varphi_{1} dx+\int_{\Omega}\nabla  w_{r}\nabla \varphi_{2} dx \\
 =\int_{\Omega}f(x,  z_{r})\varphi_{2} dx+\int_{\Omega}  w_{r}\varphi_{1}dx, \ \ \forall  \varphi_{1}, \varphi_{2}\in H_{0}^{1}(\Omega).
\end{multline*}
In view of \eqref{m_{0}} and \eqref{f_2}, the identity above makes sense. We say that a weak solution $(z, w)$ of \eqref{Sys} is not trivial if $z$ and $w$ are not both zero. From \eqref{f_1}, system \eqref{Sys} does not admit the trivial solution.

In the following we denote with   with $\|\cdot\|$ the norm in $H^{1}_{0}(\Omega)$.

\medskip

\begin{proposition}\label{p1}
If \eqref{m_{0}}-\eqref{m_{1}} and \eqref{f_1}-\eqref{f_3} hold then, for each $r\geq 0$, system \eqref{Sys} has a unique nontrivial weak solution  $( z_{r},  w_{r})$.
Even more, it is $z_{r} \in H_{0}^{1}(\Omega)\cap W^{3,2}(\Omega)\cap C^{1, \alpha}(\overline{\Omega})$, for some $\alpha\in (0, 1)$.
\end{proposition}

Although this proposition is the analogous of Proposition \ref{exis}, the proof (except the final part)
uses different arguments
and is more involved. 

\begin{proof}
%It is sufficient to prove existence and uniqueness of a weak solution $(z, w)$ to system \eqref{Sys}. 
%Here we will suppress the dependence on $r$.
Let us fix $r\geq0$. We begin by proving the existence, and then the uniqueness.
In particular in proving the existence  we will 
suppress the subscript $r$, so we will prove the existence of a solution which will be denoted
with $(\overline z, \overline w)$.

\medskip

{\bf Step 1: Existence}

\medskip

Let us fix an arbitrary $z_{1}\in H_{0}^{1}(\Omega)\cap L^{\infty}(\Omega)$ and plug it into the first equation of \eqref{Sys}. From \eqref{f_2}, by classical results there is a unique $w_{1}\in H_{0}^{1}(\Omega)\cap C(\overline{\Omega})$ such that
$$
\int_{\Omega}\nabla w_{1}\nabla\varphi dx=\int_{\Omega}f(x, z_{1})\varphi dx, \ \forall \, \varphi\in H_{0}^{1}(\Omega).
$$
Now putting this $w_{1}$ in the second equation of \eqref{Sys}, by  Proposition \ref{exis}, we get a unique $z_{2}\in H_{0}^{1}(\Omega)\cap W^{3, 2}(\Omega)\cap C^{1, \alpha}(\overline\Omega)$ such that
$$
\int_{\Omega}\nabla z_{2}\nabla \varphi dx+\int_{\Omega}M_{r}(z_{2})\varphi dx=\int_{\Omega}w_{1}\varphi, \ \forall\, \varphi\in H_{0}^{1}(\Omega).
$$
Continuing this process, we will get two sequences $\{z_{n}\}\subset H_{0}^{1}(\Omega)\cap W^{3, 2}(\Omega)\cap C^{1, \alpha}(\overline \Omega)$ and $\{w_{n}\}\subset H_{0}^{1}(\Omega)\cap C(\overline{\Omega})$ satisfying
\begin{equation}\label{conve1}
\int_{\Omega}\nabla w_{n}\nabla\varphi dx=\int_{\Omega}f(x, z_{n})\varphi dx, \ \forall \, \varphi\in H_{0}^{1}(\Omega),
\end{equation}
and
\begin{equation}\label{conve2}
\int_{\Omega}\nabla z_{n+1}\nabla \varphi dx+\int_{\Omega}M_{r}(z_{n+1})\varphi dx=\int_{\Omega}w_{n}\varphi dx, \ \forall \, \varphi\in H_{0}^{1}(\Omega).
\end{equation}

By choosing $\varphi=z_{n+1}$ in \eqref{conve2},  by the sign condition stated in Lemma \ref{le} (a), the H\"older and Poincar\'e inequalities, we get that %(hereafter $\|\cdot \|$ is the $H^{1}_{0}-$norm)
\begin{equation}\label{est1}
\|z_{n+1}\|\leq \frac{1}{\lambda_{1}}\|w_{n}\|, \ \forall \, n\in\N.
\end{equation}
On the other hand, choosing $\varphi=w_{n}$ in \eqref{conve1} and using \eqref{f_2} and \eqref{est1} we obtain
\begin{eqnarray*}
\|w_{n}\|^{2}&\leq& |\mu|_{\infty}\int_{\Omega}w_{n} dx + \nu\int_{\Omega}|z_{n}|^{\delta}w_{n} dx\\
&\leq & |\mu|_{\infty}|\Omega|^{1/2}|w_{n}|_{2} + \nu \left(\int_{\Omega}|z_{n}|^{2\delta}dx\right)^{1/2}|w_{n}|_{2}\\
&\leq& |\mu|_{\infty}|\Omega|^{1/2}|w_{n}|_{2}+\nu|\Omega|^{1-\delta}|z_{n}|_{2}^{\delta}|w_{n}|_{2}\\
&\leq& \left( \frac{|\mu|_{\infty}|\Omega|^{1/2}}{\lambda_{1}^{1/2}}+\frac{\nu|\Omega|^{1-\delta}}{\lambda_{1}^{(1+\delta)/2}}\|z_{n}\|^{\delta}\right)\|w_{n}\|\\
&\leq & \left(\frac{|\mu|_{\infty}|\Omega|^{1/2}}{\lambda_{1}^{1/2}}+\frac{\nu|\Omega|^{1-\delta}}{\lambda_{1}^{(1+3\delta)/2}}\|w_{n-1}\|^{\delta}\right)\|w_{n}\|.
\end{eqnarray*}
Therefore
\begin{equation}\label{finish}
\|w_{n}\|\leq C_{1}+C_{2}\|w_{n-1}\|^{\delta}, \ \forall \, n\geq 2,
\end{equation}
with $C_{1}=\max\{1, |\mu|_{\infty}|\Omega|^{1/2}/\lambda_{1}^{1/2}\}$ and $C_{2}=\nu|\Omega|^{1-\delta}/\lambda_{1}^{(1+3\delta)/2}$.
Observe that, from \eqref{finish},
using that $\delta\in(0,1]$ and $C_{1}\leq1,$ we have
\begin{eqnarray*}
\|w_{3}\|&\leq& C_{1}+C_{2}\left(C_{1}+C_{2}\|w_{1}\|^{\delta}\right)^{\delta}\\
&\leq& C_{1}+C_{1}C_{2}+C_{2}^{2}\|w_{1}\|^{\delta}.
\end{eqnarray*}
In an analogous way
$$
\|w_{4}\|\leq C_{1}+C_{1}C_{2}+C_{1}C_{2}^{2}+C_{2}^{3}\|w_{1}\|^{\delta},
$$
and consequently,
\begin{equation}\label{inequ}
\|w_{n}\|\leq C_{1}\sum_{k=0}^{n-2}C_{2}^{k}+ \|w_{1}\|^{\delta}C_{2}^{n-1}, \ \forall \,n\geq 2.
\end{equation}
Since $\nu\in (0, \lambda_{1}^{(1+3\delta)/2}/|\Omega|^{1-\delta})$ it follows from \eqref{inequ} that $\{w_{n}\}$, and consequently $\{z_{n}\}$ by \eqref{est1}, is bounded in $H_{0}^{1}(\Omega)$.
Hence there are $\overline w, \overline z\in H_{0}^{1}(\Omega)$ such that, passing to a subsequence, 
\begin{equation}\label{conver1}
w_{n}\rightharpoonup \overline w \ \mbox{and} \ z_{n}\rightharpoonup \overline z \ \mbox{in} \ H_{0}^{1}(\Omega),
\end{equation}
%\begin{equation}
%\textcolor{red}{tirar isso?}w_{n}\to w \ \mbox{and} \ z_{n}\to z \ \mbox{in} \ L^{2}(\Omega),
%\end{equation}
\begin{equation}\label{conver3}
w_{n}(x)\to \overline w(x) \ \mbox{and} \ z_{n}(x)\to \overline z(x) \ \mbox{a.e. in $\Omega$}
\end{equation}
and for some $g, h\in L^{2}(\Omega)$,
\begin{equation}\label{conver4}
|w_{n}(x)|\leq g(x) \ \mbox{and} \ |z_{n}(x)|\leq h(x) \ \mbox{a.e. in $\Omega$}.
\end{equation}

We are going to show that $(\overline z,\overline w)$ is the solution we were looking for.
From \eqref{conver1}, we have
\begin{equation}\label{101}
\int_{\Omega}\nabla w_{n}\nabla\varphi dx\to \int_{\Omega}\nabla \overline w\nabla\varphi dx, \ \forall\, \varphi\in H_{0}^{1}(\Omega)
\end{equation}
and
\begin{equation}\label{102}
\int_{\Omega}\nabla z_{n+1}\nabla \varphi dx \to \int_{\Omega}\nabla \overline z\nabla \varphi dx, \ \forall \, \varphi\in H_{0}^{1}(\Omega).
\end{equation}
Moreover by the Lebesgue Dominated Convergence Theorem  (possible in view of
\eqref{conver3}-\eqref{conver4}) we get
\begin{equation}\label{103}
\int_{\Omega}w_{n}\varphi dx\to \int_{\Omega}\overline w\varphi dx, \ \forall\, \varphi\in H_{0}^{1}(\Omega)
\end{equation}
and
%Since $f$ a Caratheodory function, it follows from \eqref{f_2}, \eqref{conver3}-\eqref{conver4} and Lebesgue Dominated Convergence Theorem that
\begin{equation}\label{104}
\int_{\Omega}f(x, z_{n})\varphi dx\to \int_{\Omega}f(x, \overline z)\varphi dx, \ \forall\, \varphi\in H_{0}^{1}(\Omega).
\end{equation}

From \eqref{f_2}, for each $n\in \N$, $f(., z_{n}(.))\in L^{\infty}(\Omega)$ and
$$
|f(., z_{n}(.))|_{\infty}\leq |\mu|_{\infty}+\nu|z_{n}|_{\infty}^{\delta},
$$
and then, by using  the definition of the  constant $\gamma$, it follows that
\begin{equation}\label{eq:estiman}
|w_{n}|_{\infty}\leq \gamma |f(., z_{n}(.))|_{\infty}\leq \gamma\left(|\mu|_{\infty}+\nu|z_{n}|_{\infty}^{\delta}\right).
\end{equation}
%where $\gamma$ does not depend of $n$.

Recalling Remark \ref{rem1}, we have
\begin{equation}\label{desi}
|M_{r}(z_{n+1})|\leq |w_{n}|_{\infty},
\end{equation} 
and also
\begin{equation}\label{est2}
|z_{n}|_{\infty}\leq \frac{1}{\mathfrak{m}}|w_{n-1}|_{\infty}, \ \forall \, n\geq 2
\end{equation}
which, joint with \eqref{eq:estiman} furnishes
$$
|w_{n}|_{\infty}\leq \gamma|\mu|_{\infty}+\frac{\nu \gamma}{\mathfrak{m}^{\delta}}|w_{n-1}|
_{\infty}^{\delta}\leq\tilde{\gamma}+\frac{\nu \gamma}{\mathfrak{m}^{\delta}}|w_{n-1}|
_{\infty}^{\delta}, \ 
\forall \, n\geq 2
$$
%and therefore
%$$
%|w_{n}|_{\infty}\leq \tilde{\gamma}+\frac{\nu \gamma}{\mathfrak{m}^{\delta}}|w_{n-1}|_{\infty}^{\delta}, \ \forall \ n\geq 
%2.
%$$
where we have set $\tilde{\gamma}=\max\{1, \gamma|\mu|_{\infty}\}$.

Now arguing as in \eqref{finish} we get
\begin{equation}\label{inequa}
|w_{n}|_{\infty}\leq \tilde{\gamma}\sum_{k=0}^{n-2}
\left(\frac{\nu \gamma}{\mathfrak m^{\delta}}\right)^{k}+ |w_{1}|_{\infty}^{\delta}\left(\frac{\nu \gamma}{\mathfrak 
m^{\delta}}\right)^{n-1}.
\end{equation}
Since by assumptions $\nu\in (0, \mathfrak m^{\delta}/\gamma)$, it follows by \eqref{inequa} that $\{w_{n}\}$, and 
consequently $\{z_{n}\}$ by \eqref{est2}, is bounded in $L^{\infty}(\Omega)$.

From \eqref{desi} and \eqref{inequa}, we conclude that there exists $C>0$ such that
\begin{equation}\label{inequal}
|M_{r}(z_{n})|_{\infty}\leq C, \ \forall \,n\in\N.
\end{equation}
Since $M_{r}$ is continuous, by  \eqref{conver3}, \eqref{inequal} and Lebesgue Dominated Convergence Theorem it 
follows that
\begin{equation}\label{105}
\int_{\Omega}M_{r}(z_{n+1})\varphi dx\to \int_{\Omega}M_{r}(\overline z)\varphi dx, \ \forall \, \varphi\in H_{0}^{1}(\Omega).
\end{equation}

The convergences in \eqref{conver3} and the boundedness of $\{z_{n}\}$ and $\{w_{n}\}$ in $L^{\infty}(\Omega)$ tell us that $\overline z, \overline w\in H_{0}^{1}(\Omega)\cap L^{\infty}(\Omega)$. Thereby,  from \eqref{conve1}-\eqref{conve2}, 
\eqref{101}-\eqref{104} and \eqref{105} it results that  the pair $(\overline z, \overline w)$ is a weak solution of system \eqref{Sys}.

\medskip

{\bf Step 2: Uniqueness}

\medskip

Suppose that $(z_{1}, w_{1})$ and $(z_{2}, w_{2})$ are two weak solutions of system \eqref{Sys}. Thus, for $i=1, 2$,
\begin{multline}\label{eq:uniq}
\int_{\Omega}\nabla z_{i}\nabla \varphi_{1} dx+\int_{\Omega}M_{r}(z_{i})\varphi_{1} dx+\int_{\Omega}\nabla w_{i}\nabla \varphi_{2} dx\\
=\int_{\Omega}f(x, z_{i})\varphi_{2} dx+\int_{\Omega}w_{i}\varphi_{1}dx, \  \forall \, \varphi_{1}, \varphi_{2}\in H_{0}^{1}(\Omega).
\end{multline}
Choosing $\varphi_{1}=z_{1}-z_{2}$, $\varphi_{2}=0$ in \eqref{eq:uniq} and subtracting the resulting
identities, we get
$$
\|z_{1}-z_{2}\|^{2}+\int_{\Omega}\left[ M_{r}(z_{1})-M_{r}(z_{2})\right](z_{1}-z_{2}) dx=\int_{\Omega}(w_{1}-w_{2})(z_{1}-z_{2}) dx.
$$
Since $M_{r}$ is increasing (see Lemma \ref{le}(a)) we obtain
$$
\|z_{1}-z_{2}\|^{2}\leq \int_{\Omega}|w_{1}-w_{2}||z_{1}-z_{2}| dx
$$
and consequently,
\begin{equation}\label{agora}
\|z_{1}-z_{2}\|\leq\frac{1}{\lambda_{1}}\|w_{1}-w_{2}\|.
\end{equation}
On the other hand, choosing $\varphi_{1}=0$, $\varphi_{2}=w_{1}-w_{2}$
in \eqref{eq:uniq} and subtracting the resulting
identities, we obtain
$$
\|w_{1}-w_{2}\|^{2}=\int_{\Omega}\left[ f(x, z_{1})-f(x, z_{2})\right](w_{1}-w_{2}) dx.
$$
By using hypothesis \eqref{f_3} and taking into account \eqref{agora} we get
$$
\|w_{1}-w_{2}\|^{2}\leq\theta\int_{\Omega}|w_{1}-w_{2}||z_{1}-z_{2}| dx
\leq\frac{\theta} {\lambda_{1}^{2}}\|w_{1}-w_{2}\|^{2}
$$
%Thus, from \eqref{agora} and previous inequality
%$$
%\|w_{1}-w_{2}\|\leq\frac{\theta}{\lambda_{1}^{2}}\|w_{1}-w_{2}\|.
%$$
or in other words,
\begin{equation*}\label{acabou}
\left(1-\frac{\theta}{\lambda_{1}^{2}}\right)\|w_{1}-w_{2}\|\leq 0.
\end{equation*}
Since by assumptions $\theta\in (0, \lambda_{1}^{2})$, it has to be $w_{1}=w_{2}$
and consequently $z_{1}=z_{2}$.
%, \eqref{acabou} and \eqref{agora}.

\medskip

The fact that $\overline z$ is actually in $H^{1}_{0}(\Omega)\cap W^{3,2}(\Omega)\cap C^{1,\alpha}(\overline \Omega)$ is seen  with same arguments  as in the proof of Proposition \ref{exis}.
\end{proof}

%The next proposition is a direct consequence of \eqref{f_2} and the arguments 
%are as in the proof of Proposition \ref{exis}.
%
%\begin{proposition}\label{p2}
%The weak solution $(\overline z, \overline w)$ of system \eqref{Sys} obtained in the previous Proposition \ref{p1}
%satisfies $(\overline z, \overline w)\in \left(H_{0}^{1}(\Omega)\cap W^{3,2}(\Omega)\cap C^{1, \alpha}(\overline{\Omega})\right)
%\times H_{0}^{1}(\Omega)$, for some $\alpha\in (0, 1)$. 
%\end{proposition}

As a consequence, since also in this case we have $z_{r}=u_{r}$, we have
 
\begin{proposition}
Under the conditions \eqref{m_{0}}-\eqref{m_{1}} and \eqref{f_1}-\eqref{f_3}, for each $r\geq 0$, problem \eqref{P'r} admits a unique nontrivial solution $u_{r}\in H_{0}^{1}(\Omega)\cap W^{3,2}(\Omega)\cap C^{1, \alpha}(\overline{\Omega})$, for some $\alpha\in (0, 1)$. 
\end{proposition}
%
%\begin{proof}
%It follows from Propositions \ref{p1}, \ref{p2} and from Remark \ref{110}.
%\end{proof}

The next proposition is the analogous of Proposition \ref{prop}.

\begin{proposition}\label{prop2}
If \eqref{m_{0}}-\eqref{m_{2}} and \eqref{f_1}-\eqref{f_3} hold and $u_{r}\in H_{0}^{1}(\Omega)\cap W^{3,2}(\Omega)
\cap C^{1, \alpha}(\overline{\Omega})$ is the solution of \eqref{P'r} then the map
 $$S: r\in[0,+\infty)\longmapsto \int_{\Omega}|\nabla u_{r}|^{2}dx\in \mathbb [0,+\infty)$$ is continuous, bounded and 
 has a  fixed point $r_{*}$.
\end{proposition}

\begin{proof}
Let $\{r_{n}\}$ be a sequence in $ [0,\infty)$ such that $r_{n}\to r_{\infty}$. Denoting $(u_{n}, w_{n})\in \left(H_{0}^{1}(\Omega)\cap W^{3,2}(\Omega)\cap C^{1, \alpha}(\overline{\Omega})\right)\times H_{0}^{1}(\Omega)$ the unique solution of system \eqref{Sys}
with $r=r_{n}, n\in \mathbb N$ and recalling that $z_{r}=u_{r}$, we have
\begin{equation*}\label{con1}
\int_{\Omega}\nabla w_{n}\nabla\varphi dx=\int_{\Omega}f(x, u_{n})\varphi dx, \ \forall \ \varphi\in H_{0}^{1}(\Omega)
\end{equation*}
and
\begin{equation*}\label{con2}
\int_{\Omega}\nabla u_{n}\nabla \varphi dx+\int_{\Omega}M_{n}(u_{n})\varphi dx=\int_{\Omega}w_{n}\varphi dx, \ \forall \ \varphi\in H_{0}^{1}(\Omega),
\end{equation*}
where $M_{n}:=M_{r_{n}}$. Arguing exactly as in the proof of Proposition \ref{p1} (see \eqref{conve2}-\eqref{inequ}), we conclude that $\{u_{n}\}$ and $\{w_{n}\}$ are bounded in $H_{0}^{1}(\Omega)$.

Hence, there are $w_{\#},  u_{\#}\in H_{0}^{1}(\Omega)$ such that, passing to a subsequence, 
\begin{equation*}%\label{co1}
w_{n}\rightharpoonup  w_{\#} \ \mbox{and} \ u_{n}\rightharpoonup  u_{\#}  \ \mbox{in} \ H_{0}^{1}(\Omega),
\end{equation*}
\begin{equation*}
w_{n}\to  w_{\#} \ \mbox{and} \ u_{n}\to  u_{\#} \ \mbox{in} \ L^{2}(\Omega),
\end{equation*}
\begin{equation}\label{co3}
w_{n}(x)\to  w_{\#}(x) \ \mbox{and} \ u_{n}(x)\to  u_{\#}(x) \ \mbox{a.e. in $\Omega$}
\end{equation}
and
\begin{equation*}\label{co4}
|w_{n}(x)|\leq g(x) \ \mbox{and} \ |u_{n}(x)|\leq h(x) \ \mbox{a.e. in $\Omega$},
\end{equation*}
where $g, h\in L^{2}(\Omega)$.

Again, by following the arguments in \eqref{101}-\eqref{105} and remembering that system \eqref{Sys}, with $r=r_{\infty}$, 
has a unique solution $(u_{\infty}, w_{\infty})$, we conclude that $ u_{\#}=u_{\infty}$ and $ w_{\#}=w_{\infty}$. 
Note that
\begin{equation}\label{imp}
|M_{n}(u_{n})|\leq |w_{n}|_{\infty}\leq C, \ \forall \, n\geq 0.
\end{equation}

Using again the second equation in \eqref{Sys} we get (hereafter $M_{\infty} := M_{r_{\infty}}$)
\begin{equation}\label{eq}
\|u_{n}-u_{\infty}\|^{2}=-\int_{\Omega}[M_{n}(u_{n})-M_{\infty}(u_{\infty})](u_{n}-u_{\infty})dx+\int_{\Omega}(w_{n}-w_{\infty})(u_{n}-u_{\infty}) dx.
\end{equation}
Thus,
\begin{equation}\label{eq}
\|u_{n}-u_{\infty}\|\leq \frac{1}{\sqrt{\lambda_{1}}}|M_{n}(u_{n})-M_{\infty}(u_{\infty})|_{2}+\frac{1}{\lambda_{1}}\|w_{n}-w_{\infty}\|.
\end{equation}

On the other hand, using the first equation in \eqref{Sys}, it follows that
$$
\|w_{n}-w_{\infty}\|^{2}=\int_{\Omega}\left(f(x, u_{n})-f(x, u_{\infty})\right)(w_{n}-w_{\infty}) dx
$$
and then by \eqref{f_3}, we obtain
\begin{equation}\label{abc}
\|w_{n}-w_{\infty}\|\leq\frac{\theta}{\lambda_{1}}\|u_{n}-u_{\infty}\|.
\end{equation}
Replacing \eqref{abc} in \eqref{eq}, 
\begin{equation}\label{acs}
\|u_{n}-u_{\infty}\|\leq \frac{1}{\sqrt{\lambda_{1}}}|M_{n}(u_{n})-M_{\infty}(u_{\infty})|_{2}+\frac{\theta}{\lambda_{1}^{2}}\|u_{n}-u_{\infty}\|.
\end{equation}

Therefore,
\begin{equation*}
\left(1-\frac{\theta}{\lambda_{1}^{2}}\right)\|u_{n}-u_{\infty}\|\leq
 \frac{1}{\sqrt{\lambda_{1}}}|M_{n}(u_{n})-M_{\infty}(u_{\infty})|_{2},
\end{equation*}
where $\theta\in (0, \lambda_{1}^{2})$.

From \eqref{co3} and Lemma \ref{le}(b), 
$$
|M_{n}(u_{n}(x))-M_{\infty}(u_{\infty}(x))|\to 0 \ \mbox{a.e. in $\Omega$}.
$$
On the other hand, from \eqref{imp}
$$
|M_{n}(u_{n}(x))-M_{\infty}(u_{\infty}(x))|\leq C.
$$
Therefore, by using Lebesgue Dominated Convergence Theorem we conclude from \eqref{acs} that
$$
S(r_{n})=\int_{\Omega}|\nabla u_{n}|^{2}dx\to \int_{\Omega}|\nabla u_{\infty}|^{2}dx=S(r_{\infty}).
$$
Showing that $S$ is continuous.
%Finally, for short, we use again the symbols $u_{0}$ and $w_{0}$, at this time, to denote the unique nontrivial 
%solution of system \eqref{Sys}, with $r=0$. 
%Thus %\textcolor{red}{confu\c c\~ao entre $u_{0}$....}
%$$
%S(0)=\int_{\Omega}|\nabla u_{0}|^{2}dx>0.
%$$
On the  other side, by using the sign condition,
$$
S(r)\leq \int_{\Omega}w_{r}u_{r} dx\leq\frac{1}{\lambda_{1}}\|w_{r}\|\|u_{r}\|.
$$
Thus,
\begin{equation}\label{ufa}
\|u_{r}\|\leq \frac{1}{\lambda_{1}}\|w_{r}\|.
\end{equation}
Moreover, the first equation in \eqref{Sys} and hypothesis \eqref{f_3} give
$$
\|w_{r}\|^{2}\leq\theta\int_{\Omega}|w_{r}||\mu| dx+\nu\int_{\Omega}|u_{r}|^{\delta}|w_{r}|.
$$
Using \eqref{ufa} and arguing exactly as in the proof of \eqref{finish}, we get
$$
\|w_{r}\|\leq \frac{|\mu|_{\infty}|\Omega|^{1/2}}{\lambda_{1}^{1/2}}+\frac{\nu|\Omega|^{1-\delta}}{\lambda_{1}^{(1+3\delta)/2}}\|w_{r}\|^{\delta}, \ \forall \, r\geq 0.
$$
Consequently $S$ is bounded and there is $R>0$, large enough, such that 
$S(R)<R$. 
Since
$$S(0)=\int_{\Omega}|\nabla u_{0}|^{2}dx>0,$$
the existence of a positive fixed point $r_{*}$ is guaranteed.
\end{proof}

Again to $r_{*}$ is associates $ u_{*} := u_{r_{*}}$ which gives a solution 
of  problem \eqref{general}, proving Theorem \ref{th:main2}.

%The fixed point provided in Proposition \ref{prop2} gives a solution of \eqref{general},
%completing the proof of Theorem \ref{th:main2}.

\end{document}